\documentclass[11pt,reqno]{amsart}
\usepackage{xcolor}
\usepackage[T1]{fontenc}

\usepackage{amsmath,amssymb,amsthm}
\usepackage{esint}

\usepackage[a4paper,margin=2.8cm]{geometry}

\usepackage{mleftright} 
\mleftright

\usepackage[hidelinks,bookmarksdepth=3]{hyperref}

\usepackage[backend=biber,
style=alphabetic,
sorting=nyt,
maxbibnames=10,
maxcitenames=10,
giveninits=true,
doi=false,
url=false]{biblatex}

\DeclareFieldFormat[article]{journaltitle}{#1}
\DeclareFieldFormat[book]{title}{#1}

\renewbibmacro{in:}{%
  \ifentrytype{article}{}{\printtext{\bibstring{in}\intitlepunct}}}

\numberwithin{equation}{section}

\newcommand{\R}{\mathbb{R}}
\newcommand{\Z}{\mathbb{Z}}
\newcommand{\T}{\mathbb{T}}

\DeclareMathOperator{\diver}{div}

\theoremstyle{plain}
\newtheorem{thm}{Theorem}[section]

\newtheorem{lem}[thm]{Lemma}
\newtheorem{prop}[thm]{Proposition}

\theoremstyle{definition}
\newtheorem{defn}[thm]{Definition}

\theoremstyle{remark}
\newtheorem{rem}[thm]{Remark}

\begin{document}

\numberwithin{equation}{section}

\title[Osgood meets Ambrosio-DiPerna-Lions]
{Osgood meets Ambrosio-DiPerna-Lions}

\author[G. De Philippis]{Guido De Philippis}
\address{Department of Mathematics, University of Padova, Via Trieste 63, 35121 Padova, Italy}
\email{guido.dephilippis@math.unipd.it}

\author[L. Franchi]{Leonardo Franchi}
\address{Department of Pure Mathematics and Mathematical Statistics, Centre for Mathematical Sciences, Wilberforce Road, Cambridge CB3 0WA, United Kingdom}
\email{lf511@cam.ac.uk}

\begin{abstract}
In this note we extend the Ambrosio-DiPerna-Lions theory on the well-posedness of the transport equation to the case in which the vector field satisfies an \(L^p\) Osgood condition. In particular we show that bounded distributional solutions to the transport equation are unique and thus renormalized. As opposed to the classical proof, we do not rely on the vanishing of a commutator, but we show that a weighted energy built out of the Littlewood-Paley decomposition of the solution is conserved.
\end{abstract}

\maketitle
\section{Introduction}

\subsection{Background}
Let \(b: [0,T] \times \mathbb{T}^d \to \mathbb{R}^d\) be a divergence-free vector field:
\begin{equation}\label{eq:div}
 \diver b= 0.
\end{equation}
Classical Cauchy-Lipschitz theory ensures that for \(b\in L^1([0,T]; W^{1,\infty}(\mathbb{T}^d))\), there exists a unique solution to 
\begin{equation}\label{e:ode}\tag{ODE}
           \begin{cases}
   \partial_t X(t,x)=b(t,X(t,x))
   \\
   X(0,x)=x
 \end{cases}.
\end{equation}
Via the theory of characteristics, this is equivalent to the uniqueness of measure solutions\footnote{Note that thanks to~\eqref{eq:div}, equation~\eqref{tr} is equivalent to the continuity equation   \begin{equation}\label{cont}\tag{C}
\begin{cases}
 \partial_t u + \diver (bu) = 0, \\
 u(0, x) = u_0(x).
\end{cases}
\end{equation}
In particular it makes sense to speak about distributional solutions.} to the transport equation
  \begin{equation}\label{tr}\tag{T}
  \begin{cases}
 \partial_t u + b \cdot \nabla u = 0, \\
 u(0, x) = u_0(x).
\end{cases}
\end{equation}
In their seminal work~\cite{DiPernaLions1989}, DiPerna and Lions began to investigate what happens when the assumption \(\nabla b \in L^\infty\) is replaced by \(\nabla b \in L^p\), \(p\in (1,+\infty)\). In particular, they showed that under this assumption, there is still uniqueness of \emph{bounded} solutions of equation~\eqref{tr}.
Later, Ambrosio extended the theory to \(BV\) vector fields~\cite{Ambrosio2004}. He also introduced the concept of Regular Lagrangian Flow (RLF), which is, roughly speaking,  a  selection of solutions to~\eqref{e:ode}. In particular  he showed that, in general, the uniqueness of bounded solutions of~\eqref{tr} (or more precisely of~\eqref{cont}) implies  the existence of a unique Regular Lagrangian Flow.
Subsequently, Crippa and De Lellis~\cite{CrippaDeLellis2008} showed that the existence and uniqueness of a Regular Lagrangian Flow for Sobolev vector fields can be obtained directly via a Lagrangian approach. Note, however, that  uniqueness of the RLF is not known to imply uniqueness for solutions to the transport equation~\eqref{tr}, which  is a stronger requirement. 

\subsection{Main results}
It is a classical fact that uniqueness for~\eqref{e:ode} can also be obtained if  the Lipschitz assumption on \(b\)  is relaxed to an \emph{Osgood assumption}:
\begin{equation}\label{e:osgood}
 |b(t,x)-b(t,y)|\le s(t)\omega(|x-y|)
\end{equation}
where \(s\in L^1_t\) and \(\omega:\mathbb{R}_{\ge 0}\to \mathbb{R}_{\ge 0}\) is an \emph{Osgood modulus of continuity}, i.e., a continuous, increasing, concave function such that
\begin{equation}\label{osgcon}
\int_0^1 \frac{ds}{\omega(s)}=+\infty.
\end{equation}
The typical example is \(\omega(t)\sim t|\log t|\). Under  this assumption one can also show uniqueness of solutions to~\eqref{tr},~\cite{AmbrBern08,BahouriCheminDanchin2011}.

Note that~\eqref{e:osgood} can be thought of as saying that \(b\) has \(\omega\) derivatives in 
\(L^1_t(L_x^\infty)\). It is thus natural to investigate the extent to which the Ambrosio-DiPerna-Lions theory can be extended to \(\omega\) derivatives in \(L^p\).
A first attempt in this direction was made by Li and Luo~\cite{LiLuo2015}, where the Crippa-De Lellis Lagrangian approach was extended to the case where \(b\) satisfies an Osgood-Sobolev condition in the spirit of~\cite{Hajlasz1996}:
\begin{equation}\label{e:bosgoodinfinty}
|b(t,x)-b(t,y)|\le \omega(|x-y|)(g(t,x)+g(t,y))
\end{equation}
with  \(g\in L^1([0,T]; L^p(\mathbb{T}^d))\), and \(\omega\) an Osgood modulus.
Under this assumption,  the existence of a unique RLF solving~\eqref{e:ode} was shown. However, as already mentioned, this does not immediately imply uniqueness for solutions of the transport equation~\eqref{tr}.

In this paper, we directly study solutions of~\eqref{tr} and show that, under the same regularity assumption,  they are unique.

Since  the failure of the Osgood condition already implies non-uniqueness at the level of the ODE, in general one should not expect uniqueness of~\eqref{tr} when \(b\) has  ``non Osgood'' derivatives. For instance in~\cite{DiPernaLions1989}   uniqueness is shown to fail for vector fields in \(W^{s,1}\) and by  a careful tracking of the parameters, one can show that the vector field  constructed in~\cite{Depauw2003} belongs to \(L^1([0,1]; W^{s,p})\) for all \(s<1\) and  \(p<\frac{1}{s}\).

\begin{thm}\label{renormalizationosgood}
 Let \(p\in(1,\infty)\), let \(\omega\) be an Osgood modulus of continuity. Assume moreover that  there are constants \(C\) and \(\beta>0\) such that
\begin{equation}\label{omegabounds}
    \omega(\lambda s)\leq C\lambda^\beta \omega(s) \qquad \text{for every } s,\lambda\in(0,1].
\end{equation}
Let  \(b\in L^\infty([0,T]\times\T^d;\R^d)\) be divergence-free. Assume that
\[
\int_0^T \|b(t,\cdot)\|_{\dot M^{\omega,p}}dt<\infty\,,
\]
where for a measurable function \(f\) we have set
\[
\|f\|_{\dot M^{\omega,p}}
:=
\inf_g \|g\|_{L^p(\T^d)},
\]
and  the infimum is taken over all nonnegative \(g\in L^p(\T^d)\) such that
\[
|f(x)-f(y)|
\leq
\omega(|x-y|)\bigl(g(x)+g(y)\bigr).
\]
Then for every  bounded initial datum \(u_0\) there exists a unique bounded distributional solution \(u\) to  the transport equation~\eqref{tr}. In particular \(u\) is renormalized: \(g(u)\) is a solution for all \(g \in C^1(\mathbb{R};\mathbb{R})\).  
\end{thm}
Note that  the requirement~\eqref{omegabounds} is not really restrictive since it  roughly requires that \(\omega(s)\lesssim s^\beta\) for \(\beta<1\), and it is satisfied by all relevant examples.

Besides uniqueness of solutions to the transport equation, a great deal of research has been devoted recently to understanding which regularity properties of the initial datum \(u_0\) are propagated by the flow. This is particularly relevant in connection with Bressan's mixing conjecture. In particular it is known that no  fractional Sobolev regularity of \(u_0\) is  preserved by~\eqref{tr},~\cite{AlbertiCrippaMazzucato2019Loss}. On the other hand, when \(b \in W^{1,p}\), some  \(\log\)-regularity is preserved. This was shown for the first time in~\cite{Lege18}, based on~\cite{HadzSeegSmarStre18,SeegSmarStre19}, via an Eulerian technique.  The optimal regularity was then established in~\cite{BrueNguyen2021} via a Lagrangian technique and later in~\cite{meyer2024propagation} via Littlewood-Paley theory. See also~\cite{HuysSaid26} and~\cite{BrueColoPhilJoha26}. Our second main result establishes the optimal propagation of regularity in the Osgood case. See also~\cite[Theorem 4.1]{LiLuo2015} for the Lagrangian approach. 

\begin{thm}\label{regularityosgood}
Let \(p\in(1,\infty)\), let \(\omega\) be an Osgood modulus of continuity that satisfies~\eqref{omegabounds}, and let
\(b\in L^\infty([0,T]\times\mathbb{T}^d;\mathbb{R}^d)\) be divergence-free.
Assume that
\(
\int_0^T
\|b(r,\cdot)\|_{\dot M^{\omega,p}}\,dr
<\infty.
\)

Let \(u\in L^\infty([0,T]\times \mathbb{T}^d)\) be a bounded solution to~\eqref{tr} with initial datum \(u_0\). If
\(u_0\in B^{\omega,a}(\mathbb{T}^d)\) for some
\(\frac12\leq a\leq\frac p2\), then, for every \(t\in[0,T]\),
\[
\|u(t)\|_{B^{\omega,a}}
\lesssim
\left(
\int_0^t
\|b(r,\cdot)\|_{\dot M^{\omega,p}}\,dr
\right)^a
\|u\|_{ L^\infty([0,T]\times \mathbb{T}^d)}
+
\|u_0\|_{B^{\omega,a}},
\]
with the implicit constant depending only on \(d\), \(p\), \(\omega\), and \(a\).
\end{thm}

In the above theorem, the Besov-like space \(B^{\omega,a}\) is defined in terms of \(\omega\) (or more precisely of a suitable primitive of its reciprocal) and coincides with the logarithmic spaces when \(\omega(t)\sim t\).

\subsection{Ideas of the proof}

The classical approach to uniqueness of solutions to~\eqref{tr} is to show that solutions are renormalized. In order to do this, one  shows that  the mollification of a solution \(u_\delta=u*\varphi_{\delta}\) satisfies an equation of the form
\[
 \partial_t u_\delta+b\cdot\nabla u_\delta=r_{\delta}
\]
where \(r_{\delta }\) is the commutator:
\[
 r_{\delta }=b\cdot\nabla u_\delta-(\diver (bu))_\delta.
\]
Since, dimensionally, the commutator scales like \(u\nabla b\), one then relies on the Sobolev regularity of \(b\) to show that \(r_\delta\) vanishes as \(\delta \to 0\).

In our case we shall argue differently:  indeed a priori it is not even clear that under our Osgood assumption \(r_\delta\) is bounded. Our proof is  based on the  ``commutator'' between a suitable pseudo-differential operator and  the transport operator \( b\cdot\nabla\). 

More precisely, we let  \(u_k\) be the \(k\)-th Littlewood-Paley piece of \(u\)  (see the next section) and for a smooth function \(\chi\), let \(\chi(D)\) be the operator defined as
\[
\chi(D)u=\sum_{k}\chi(k)u_k.
\]
One formally  gets that
\[
\partial_t (\chi(D)u)+b\cdot\nabla\chi(D)u = [b\cdot \nabla , \chi(D)]u.  
\]
By exploiting the antisymmetry of \(b\cdot\nabla\),  one gets that (roughly and neglecting various other terms)
\[
 \frac{d}{dt} \frac{1}{2} \sum_{k}\chi(k)\|u_k\|^2_{L^2} =\bigl\langle [b\cdot \nabla , \chi(D)]u, u\bigr\rangle_{L^2}\lesssim \sum_{k}\chi'(k) \int |\nabla  b_{\le k}| u^2_k  .
\]
The presence of the derivative of \(\chi\) is the key fact that allows for a suitably chosen sequence of test functions \(\chi_R\to 1 \) such that the right-hand side vanishes. This shows the conservation of the \(L^2\) norm of \(u\) for any bounded solution, which easily implies uniqueness. The proof of regularity then follows the same lines. In particular one should note that when \(b\) is Sobolev and \(\chi (t)\sim t\), the above calculation is the underlying mechanism that allows for the propagation of the square root of the log of a derivative. 

We conclude by noticing that the commutator between the transport operator and various pseudo-differential operators was already considered in~\cite{HuysSaid26}. There the main ``cancellation'' phenomenon was based on the commutator  estimates of~\cite{HadzSeegSmarStre18}. While the underlying phenomenon is the same, the estimates do not seem to be  localized enough in frequency space to be able to obtain the bounds needed to deal with Osgood regularity.

\subsection*{Acknowledgments}
GDP's research is funded by the European Research Council (ERC) through CoG   101169953 ``RISE''.\footnote{Views and opinions expressed are however those of the authors only and do not necessarily reflect those of the European Union or the European Research Council.}
LF acknowledges support from the Isaac Newton Trust through a Trinity Cambridge Research Studentship. LF also acknowledges the support of the Scuola Galileiana di Studi Superiori, where part of this work was carried out during his time as a student.
The authors used ChatGPT and Claude for minor copyediting assistance during the final revision of this paper. The manuscript and all mathematical content were written entirely by the authors.

\section{Littlewood-Paley decomposition and preliminary estimates}

Throughout the paper, we write \(A\lesssim B\) if there exists a constant \(C>0\)
such that \(A\le CB\). We write \(A\sim B\) if \(A\lesssim B\) and \(B\lesssim A\).
Unless otherwise stated, the implicit constants may depend on the dimension and
on other fixed parameters.
\subsection{Fourier preliminaries}

We fix our Fourier conventions on \(\mathbb{T}^d\) and \(\mathbb{R}^d\).
For \(k \in \mathbb{Z}^d\), we define the Fourier coefficients of a function \(f\) on
\(\mathbb{T}^d=\mathbb{R}^d/\mathbb{Z}^d\) by
\[
\widehat{f}(k)=\int_{\mathbb{T}^d} f(x)e^{-2\pi i k\cdot x}\,dx.
\]
For \(\xi \in \mathbb{R}^d\), we define the Fourier transform of a function \(K\) on
\(\mathbb{R}^d\) by
\[
\widehat{K}(\xi)=\int_{\mathbb{R}^d} K(x)e^{-2\pi i \xi\cdot x}\,dx.
\]
Whenever we convolve a periodic function \(f\) on \(\mathbb{T}^d\) with a function
\(K\) on \(\mathbb{R}^d\), we regard \(f\) as its periodic extension to \(\mathbb{R}^d\). Accordingly,
\begin{equation}\label{eq:convolution}
(K*f)(x)
:= \int_{\mathbb{R}^d} K(x-y)f(y)\,dy
= \int_{\mathbb{T}^d} \sum_{m\in\mathbb{Z}^d} K(x-z-m)\,f(z)\,dz,
\end{equation}
which defines a periodic function on \(\mathbb{T}^d\). Whenever the expressions are well
defined, for every \(k\in\mathbb{Z}^d\) we then have
\begin{equation}\label{eq:fourier_convolution_product}
\widehat{K*f}(k)=\widehat{K}(k)\widehat{f}(k).
\end{equation}
Similarly, for functions \(f,g\) on \(\mathbb{T}^d\) and every \(k\in\mathbb{Z}^d\),
\begin{equation}\label{eq:fourier_product_convolution}
\widehat{fg}(k)=\sum_{\ell\in\mathbb{Z}^d}\widehat{f}(\ell)\widehat{g}(k-\ell).
\end{equation}
For functions \(f,g\in L^2(\mathbb{T}^d)\), Parseval's identity reads
\begin{equation}\label{eq:parseval}
\int_{\mathbb{T}^d} f(x)\overline{g(x)}\,dx
=
\sum_{k\in\mathbb{Z}^d}\widehat{f}(k)\overline{\widehat{g}(k)}.
\end{equation}

\subsection{Littlewood-Paley decomposition}

We recall the Littlewood-Paley decomposition on \(\mathbb{T}^d\); for a more detailed
treatment we refer to~\cite[Chapter~6]{grafakos2014classical}
and~\cite[Chapter~2]{BahouriCheminDanchin2011}.

Let \(\varphi\in\mathcal{S}(\mathbb{R}^d)\) be radial. We choose \(\widehat{\varphi}\) to be radial and non-increasing as a function of \(|\xi|\), with
\[
0\leq \widehat{\varphi}\leq1,
\qquad
\operatorname{spt}\widehat{\varphi}\subset B_1(0),
\qquad
\widehat{\varphi}=1 \quad\text{on }\overline{B_{1/2}(0)}.
\]
We further choose \(\varphi\) so that there exists a smooth real-valued radial function \(\vartheta\) satisfying
\[
\vartheta(\xi)^2
=
\widehat{\varphi}(\xi)-\widehat{\varphi}(2\xi)
\qquad\text{for every }\xi\in\mathbb{R}^d.
\] We set
\(\varphi_0:=\varphi\) and, for each integer \(k\geq 1\), we define
\begin{equation}\label{pezzilittlewoodpaley}
\varphi_k(x):=2^{kd}\varphi(2^k x)-2^{(k-1)d}\varphi(2^{k-1}x).
\end{equation}
Then
\begin{equation}\label{pezzitrasformati}
\widehat{\varphi_k}(\xi)
=
\widehat{\varphi}\!\left(\frac{\xi}{2^k}\right)
-
\widehat{\varphi}\!\left(\frac{\xi}{2^{k-1}}\right),
\end{equation}
for every \(k\geq 1\) and every \(\xi\in\mathbb{R}^d\). Since \(\widehat{\varphi}\) is radial and non-increasing as a function of \(|\xi|\), it follows that
\begin{equation}\label{eq:phikpositive}
\widehat{\varphi_k}(\xi)\geq 0
\qquad\text{for every } k\geq 1 \text{ and every } \xi\in\mathbb{R}^d.
\end{equation}
Furthermore,
\begin{equation}\label{eq:dyadic_support}
\operatorname{spt}\widehat{\varphi_k}
\subset
B_{2^k}(0)\setminus \overline{B_{2^{k-2}}(0)},
\end{equation}
for every \(k\geq 1\). Hence \(\widehat{\varphi_k}\) and \(\widehat{\varphi_\ell}\) have disjoint supports whenever \(|k-\ell|\geq 2\). Moreover,
\[
\widehat{\varphi_k}
=
\widehat{\varphi_k}\bigl(\widehat{\varphi}_{k-1}+\widehat{\varphi}_k+\widehat{\varphi}_{k+1}\bigr),
\]
for every \(k\geq 1\), where for \(k=1\) we understand \(\varphi_{k-1}\) as \(\varphi_0=\varphi\).
Taking inverse Fourier transforms, we obtain
\begin{equation}\label{relazionepezzinontrasformati}
\varphi_k
=
\varphi_k * (\varphi_{k-1}+\varphi_k+\varphi_{k+1}),
\end{equation}
for every \(k\geq 1\).
For a function \(f\) on \(\mathbb{T}^d\), we define
\(
f_k:=f*\varphi_k,
\)
for every \(k\geq 0\). Since \(B_1(0)\cap \mathbb{Z}^d=\{0\}\), the function \(f_0=f*\varphi\) is constant and coincides with the average of \(f\) over \(\mathbb{T}^d\). For each integer \(k\geq 0\), we also set
\(
\psi_k:=\sum_{j=0}^k \varphi_j.
\)
By telescoping we have
\[
\psi_k(x)=2^{kd}\varphi(2^k x)
\qquad\text{and}\qquad
\widehat{\psi_k}(\xi)=\widehat{\varphi}\!\left(\frac{\xi}{2^k}\right)
\]
for every \(k\geq 0\) and every \(\xi\in\mathbb{R}^d\). The first identity above, together with the change of variables \(z=2^k y\) and the rapid decay of \(\varphi\), gives the useful estimate
\[
\int_{\mathbb{R}^d}(1+2^k|y|)^N
\left(
|\psi_k(y)|
+2^{-k}|\nabla\psi_k(y)|
+|y||\nabla\psi_k(y)|
\right) dy
\lesssim_N 1
\]
for every \(N>0\). We will use this estimate several times in Section~3. We then define
\[
f_k^{\le}:=f*\psi_k \qquad \text{for every } k\geq 0,
\qquad\text{and}\qquad
f_k^{\ge}:=f-f*\psi_{k-1} \qquad \text{for every } k\geq 1.
\]
Thus \(f=f_k^{\le}+f_{k+1}^{\ge}\) for every \(k\geq 0\). Finally, by telescoping,
\[
\widehat{\varphi}(\ell)+\sum_{k=1}^\infty \widehat{\varphi_k}(\ell)=1
\]
for every \(\ell \in\mathbb{Z}^d\). As a consequence, if \(1<p<\infty\) and \(f\in L^p(\mathbb{T}^d)\), then
\[
f=f_0+\sum_{k=1}^\infty f_k
\qquad\text{and}\qquad
f_k^{\ge}=\sum_{j=k}^\infty f_j \quad \text{for every } k\geq 1
\]
with convergence in \(L^p(\mathbb{T}^d)\). In particular, if \(f\) is mean-free, then
\(f=\sum_{k=1}^\infty f_k\) in \(L^p(\mathbb{T}^d)\).

We shall also use the corresponding quadratic Littlewood-Paley decomposition. For every \(k\geq1\), set \(\vartheta_k(\xi):=\vartheta(\xi/2^k)\), and let \(\theta_k\in\mathcal S(\mathbb R^d)\) be defined by \(\widehat{\theta_k}=\vartheta_k\). Then \(\vartheta_k(\xi)^2=\widehat{\varphi_k}(\xi)\) and \(\sum_{k\geq1}\vartheta_k(\xi)^2=1\) for every \(|\xi|\geq1\). We set \(\widetilde f_k:=f*\theta_k\). By construction, \(\varphi_k=\theta_k*\theta_k\) and hence \(f_k=f*\theta_k*\theta_k\). In particular, if \(f\) is mean-free, Parseval's identity gives
\[
\sum_{k\geq1}\|\widetilde f_k\|_{L^2(\T^d)}^2
=
\|f\|_{L^2(\T^d)}^2.
\]
\subsection{Classical Littlewood-Paley results}
We begin by recalling the classical square-function estimate in the mean-free case. If \(u\) is mean-free, then for every
\(q\in(1,\infty)\) one has
\begin{equation}\label{LPtheorem}
\left\|
\left(
\sum_{k\geq 1} |u_k|^2
\right)^{1/2}
\right\|_{L^q}
\sim
\|u\|_{L^q}.
\end{equation}
We refer to~\cite[Theorem~6.1.2]{grafakos2014classical} for a proof.
We will also need the upper bound for arbitrary dyadic dilates of a mean-zero Schwartz function, which is also contained in~\cite[Theorem~6.1.2]{grafakos2014classical}. More precisely, if \(\widetilde{\varphi}_k=2^{kd}\widetilde{\varphi}_0(2^k\cdot)\), where \(\widetilde{\varphi}_0\) is a mean-zero Schwartz function, then for every \(q\in(1,\infty)\),
\begin{equation}\label{LPupperbound}
\left\|
\left(
\sum_{k\geq 1} |u * \widetilde{\varphi}_k|^2
\right)^{1/2}
\right\|_{L^q}
\lesssim
\|u\|_{L^q}.
\end{equation}

We next recall two estimates involving the Hardy-Littlewood maximal operator that will be used together with the square-function bounds above. The first is the classical Fefferman-Stein inequality~\cite{FeffermanStein1971}. Here and below, \(M\) denotes the Hardy-Littlewood maximal operator.
\begin{thm}[Fefferman-Stein inequality]\label{FSvector}
Let \(1<p<\infty\). Then, for every sequence
\(\{F_j\}_{j\geq1}\subset L^p(\mathbb{T}^d)\),
\begin{equation}\label{FSvectorineq}
\left\|
\left(
\sum_{j\geq1}|M(F_j)|^2
\right)^{1/2}
\right\|_{L^p(\mathbb{T}^d)}
\lesssim_p
\left\|
\left(
\sum_{j\geq1}|F_j|^2
\right)^{1/2}
\right\|_{L^p(\mathbb{T}^d)}.
\end{equation}
\end{thm}

The second is a pointwise estimate for frequency-localized functions, which is a standard consequence of Peetre's maximal inequality; see, for example,~\cite[Lemma~14.6.1]{hytonenvanneervenveraarweis2023}.

\begin{lem}
Let \(C>0\) and \(j\in\mathbb{Z}\), and let \(h\) be a smooth
function on \(\mathbb{R}^d\) such that
\[
\operatorname{spt}\widehat{h}\subset B_{C2^j}(0).
\]
Then, for every \(N>d\),
\begin{equation}\label{eq:maximaltraslate}
|h(x-y)|+2^{-j}|\nabla h(x-y)|
\lesssim_{N,C}
(1+2^j|y|)^N Mh(x)
\end{equation}
for every \(x,y\in\mathbb{R}^d\).
\end{lem}

\subsection{Weighted Littlewood-Paley spaces}

We now introduce the weighted Littlewood-Paley spaces in which we will propagate regularity for the transport equation. Let \(\omega\) be an Osgood modulus of continuity. For every \(x>0\), set
\[
G(x):=\int_{2^{-x}}^1 \frac{ds}{\omega(s)},
\]
and define \(G(x):=0\) for \(x\le 0\).
Since \(\omega\) is positive on \((0,1]\) and satisfies the Osgood condition, the function \(G\)
is increasing and
\[
\lim_{x\to\infty} G(x)=+\infty.
\]
In the model case \(\omega(s)=s\), one has \(G(k)\sim k\).
Let now \(a\geq 0\). For a function \(f\) on \(\mathbb{T}^d\), we define
\begin{equation}\label{defbwnorma}
\|f\|_{B^{\omega,a}}
:=
\left(
\sum_{k=1}^{\infty} G(k)^{2a}\|f_k\|_{L^2}^2
\right)^{1/2}.
\end{equation}
When \(f\) is mean-free, this defines a norm on that subspace of
\(L^2(\mathbb{T}^d)\). In general, it is only a seminorm, since the constant part of \(f\)
does not contribute. We denote by \(B^{\omega,a}(\mathbb{T}^d)\) the collection of all
functions \(f\in L^2(\mathbb{T}^d)\) such that \(\|f\|_{B^{\omega,a}}<\infty\).
In the model case \(\omega(s)=s\), the weight \(G(k)^{2a}\) is comparable to \(k^{2a}\), and
therefore
\[
\|f\|_{B^{\omega,a}}
\sim
\left(
\sum_{k=1}^{\infty} k^{2a}\|f_k\|_{L^2}^2
\right)^{1/2}.
\]
Thus \(B^{\omega,a}\) reduces to the logarithmic weighted space used by Meyer and
Seis in their Littlewood-Paley proof of the sharp logarithmic regularity
estimate~\cite{meyer2024propagation}. This is the Littlewood-Paley counterpart
of the logarithmic regularity scale appearing in the estimates of Bru\'e and
Nguyen~\cite{BrueNguyen2021}. See also~\cite{BrueNguyenLogSobolev} for the
corresponding logarithmic Sobolev spaces.

We will also need the following simple properties of the weight \(G\).

\begin{lem}\label{lemma gk+1}
For every integer \(k\geq 1\), one has
\[
G(k)\leq G(k+1)\leq G(k)+G(1).
\]
In particular, \(G(k)\sim G(k+1)\), uniformly in \(k\).
\end{lem}

\begin{proof}
The first inequality is immediate. For the second one, we write
\[
G(k+1)=G(k)+\int_{2^{-(k+1)}}^{2^{-k}} \frac{ds}{\omega(s)}.
\]
Using the concavity of \(\omega\) and the fact that \(\omega(0)=0\), one has
\[
\omega(\lambda s)\geq \lambda \omega(s)
\]
for every \(s\geq 0\) and every \(\lambda\in[0,1]\). Taking \(\lambda=2^{-k}\), we obtain
\[
\omega\left(\frac{s}{2^k}\right)\geq \frac{1}{2^k}\omega(s)
\]
for every \(s\in[1/2,1]\). Hence
\begin{align*}
\int_{2^{-(k+1)}}^{2^{-k}} \frac{ds}{\omega(s)}
&=
\int_{1/2}^1 \frac{2^{-k}\,dr}{\omega(r/2^k)}
\leq
\int_{1/2}^1 \frac{dr}{\omega(r)}
=
G(1),
\end{align*}
which proves the claim.
\end{proof}
\begin{rem}
The same concavity argument shows that, for every integer \(k\geq1\),
\[
\omega(2^{-(k+1)})
\leq
\omega(2^{-k})
\leq
2\omega(2^{-(k+1)}).
\]
Iterating this estimate, we obtain, for every fixed \(L\geq1\),
\begin{equation}\label{eq:adjacentomega}
\omega(2^{-j})\sim_L\omega(2^{-k})
\qquad
\text{whenever } |j-k|\leq L,
\end{equation}
uniformly in \(j\) and \(k\).
\end{rem}

\begin{lem}
For every integer \(k\geq1\), one has
\begin{equation}\label{incrementG}
G(k+1)-G(k)
\sim
\frac{2^{-k}}{\omega(2^{-k})}.
\end{equation}
\end{lem}

\begin{proof}
By the monotonicity of \(\omega\),
\[
\frac{2^{-(k+1)}}{\omega(2^{-k})}
\leq
G(k+1)-G(k)
\leq
\frac{2^{-(k+1)}}{\omega(2^{-(k+1)})}.
\]
The conclusion follows from~\eqref{eq:adjacentomega}.
\end{proof}

We end this section with a lemma that bounds the supremum of some weighted quantities involving the Littlewood-Paley blocks, where the weights depend on the Osgood modulus
of continuity.

\begin{lem}\label{hajlaszdyadicbounds}
Let \(p\in(1,\infty)\) and let
\(f\in L^\infty(\mathbb{T}^d;\mathbb{R}^m)\) satisfy
\(\|f\|_{\dot M^{\omega,p}}<\infty\). Then
\begin{equation}\label{eq:hajlaszdyadicbounds}
\left\|
\sup_{k\geq1}
\frac{|f_k|+2^{-k}|\nabla f_k^{\le}|}
{\omega(2^{-k})}
\right\|_{L^p}
\lesssim_p
\|f\|_{\dot M^{\omega,p}}.
\end{equation}
\end{lem}
\begin{proof}
Let \(g\in L^p(\mathbb{T}^d)\) be a nonnegative function such that
\[
|f(x)-f(y)|
\leq
\omega(|x-y|)\bigl(g(x)+g(y)\bigr)
\]
for almost every \(x,y\in\mathbb{T}^d\). By~\eqref{pezzilittlewoodpaley},
\(\varphi_k\) has zero integral, and hence
\[
|f_k(x)|
\leq
\int_{\mathbb{R}^d}
|\varphi_k(y)|\,|f(x-y)-f(x)|\,dy.
\]
It follows from the definition of \(g\) that
\[
|f_k(x)|
\leq
\int_{\mathbb{R}^d}
|\varphi_k(y)|\omega(|y|)
\bigl(g(x-y)+g(x)\bigr)\,dy.
\]
Since \(\omega\) is concave and monotone we have
\[
\omega(|y|)
\leq
(1+2^k|y|)\omega(2^{-k}).
\]
Moreover, since \(\varphi\) is a Schwartz function,~\eqref{pezzilittlewoodpaley}
gives
\(
|\varphi_k(y)|
\lesssim
\frac{2^{kd}}{(1+2^k|y|)^{d+2}}.
\)
Combining the previous two estimates and splitting the integral into
dyadic annuli, we obtain
\begin{align*}
|f_k(x)|
&\lesssim
\omega(2^{-k})
\left(
g(x)
+
\sum_{j=0}^{\infty}
2^{kd-j(d+1)}
\int_{|y|\leq 2^{j+1-k}}g(x-y)\,dy
\right) \\
&\lesssim
\omega(2^{-k})
\left(
g(x)+\sum_{j=0}^{\infty}2^{-j}Mg(x)
\right)
\lesssim
\omega(2^{-k})Mg(x).
\end{align*}
The gradient estimate follows in the same way by using the mean-zero kernel \(2^{-k}\nabla\psi_k(y)\) to subtract \(f(x)\) inside the convolution. Thus,
\[
\sup_{k\geq1}
\frac{|f_k|+2^{-k}|\nabla f_k^{\le}|}
{\omega(2^{-k})}
\lesssim
Mg
\]
almost everywhere. Taking the \(L^p\) norm, using the boundedness of the
Hardy-Littlewood maximal operator, and then taking the infimum over all
admissible functions \(g\) proves the claim.
\end{proof}

\section{Proofs of Theorems~\ref{renormalizationosgood} and~\ref{regularityosgood}}
In this section, we prove uniqueness and propagation of regularity for solutions associated with vector fields satisfying our assumptions.

Let us recall the definition of a weak solution in our setting.
\begin{defn}
We say that a locally summable function
\(u:[0,T]\times\mathbb{T}^d\to\mathbb{R}\) is a solution of the Cauchy problem
for the transport equation with initial datum
\(u_0\in L^\infty(\mathbb{T}^d)\) if it satisfies the system
\begin{equation}\label{e:cont}
 \begin{cases}
   \partial_t u+\nabla\cdot(bu)=0,\\
   u(0,\cdot)=u_0,
 \end{cases}
\end{equation}
in the distributional sense, that is, if for every test function
\(\xi\in C_c^\infty([0,T)\times\mathbb{T}^d)\) one has
\begin{equation}\label{condistrtoro}
\int_0^T\int_{\mathbb{T}^d}
u\left[\partial_t\xi+b\cdot\nabla\xi\right]\,dx\,dt
=
-\int_{\mathbb{T}^d} u_0(x)\xi(0,x)\,dx.
\end{equation}
\end{defn}
Note that, since \(\diver b=0\), \eqref{e:cont} is equivalent to \eqref{tr}.

We start by assuming that the solution \(u\) is mean-free. This is not restrictive for both theorems. Indeed, since \(b\) is divergence-free, the spatial mean of any distributional solution is preserved in time. Testing the equation with a function of the form \(\xi(t,x)=\eta(t)\), where \(\eta\in C_c^\infty([0,T))\), shows that the distributional time derivative of
\[
t\mapsto\int_{\mathbb{T}^d}u(t,x)\,dx
\]
vanishes. Hence,
\[
\int_{\mathbb{T}^d}u(t,x)\,dx
=
\int_{\mathbb{T}^d} u_0(x)\,dx
\qquad\text{for all }t\in[0,T].
\]
Hence one may always subtract the conserved mean and reduce to the case of mean-free solutions.

We begin with a computation common to both theorems, which will then be specialized to prove the two results. We focus on estimating the time derivative of a weighted sum of the \(L^2\) norms of the quadratic Littlewood-Paley blocks. Let \(F\) be a \(C^1\) function and consider
\[
E_F(t)
:=
\sum_{k\geq1}F(G(k))
\|\widetilde u_k(t)\|_{L^2(\T^d)}^2.
\]
Since \(u\) solves the transport equation in the sense of distributions,
convolving the equation with \(\theta_k\) gives
\begin{equation}\label{eq:weighted-regularized-equation}
\partial_t\widetilde u_k+b\cdot\nabla\widetilde u_k=r_k
\end{equation}
in the sense of distributions, where
\[
r_k(t,x)
=
\int_{\R^d}
u(t,x-y)\bigl(b(t,x)-b(t,x-y)\bigr)\cdot\nabla\theta_k(y)\,dy.
\]
The term \(r_k\) is the classical DiPerna-Lions commutator~\cite{DiPernaLions1989}.

We consider the representative of \(\widetilde u_k\) which is absolutely
continuous in time; see, for instance,~\cite{DiPernaLions1989,Ambrosio2004}.
Testing~\eqref{eq:weighted-regularized-equation} against \(\widetilde u_k\) gives,
for almost every \(t\in(0,T)\),
\[
\frac12\frac d{dt}
\|\widetilde u_k(t)\|_{L^2}^2
=
\int_{\T^d}\widetilde u_k(t,x)r_k(t,x)\,dx.
\]
We will now use this identity to differentiate \(E_F(t)\). Without any further properties of the solution \(u\), this is justified only when finitely many of the terms \(F(G(k))\) are nonzero. This will be the case for our first choice of \(F\), while for the second choice we will overcome this issue using the consequences of the first one.
Thus, when the above differentiation is justified, we have
\[
\frac12\frac d{dt}E_F(t)
=
\sum_{k\geq1}
F(G(k))
\int_{\T^d}\widetilde u_k(t,x)r_k(t,x)\,dx.
\]
Moreover, summation by parts gives
\begin{equation}\label{eq:weighted-abel}
\frac12\frac d{dt}E_F(t)
=
\sum_{n\geq1}
\bigl(F(G(n))-F(G(n+1))\bigr)
\sum_{k=1}^n
\int_{\T^d}\widetilde u_k r_k\,dx.
\end{equation}
We will use~\eqref{eq:weighted-abel} to take advantage of the decay of
\(F(G(n))-F(G(n+1))\). To estimate the partial sums
\(\sum_{k=1}^n \int_{\T^d}\widetilde u_k r_k\,dx\), we first analyze the frequency interactions arising from the Littlewood-Paley decomposition. From the explicit expression for \(r_k\) and the fact that \(\operatorname{div}b=0\), we have
\[
\int_{\T^d}\widetilde u_k r_k\,dx
=
-
\int_{\T^d}\int_{\R^d}
\widetilde u_k(x)u(x-y)b(x-y)\cdot\nabla\theta_k(y)\,dy\,dx.
\]
Since \(\widetilde u_k=u*\theta_k\) and \(\theta_k\) is even, we may move the
convolution with \(\theta_k\) onto the second factor. Summing over \(k\) and using
\[
\sum_{k=1}^n\nabla(\theta_k*\theta_k)
=
\sum_{k=1}^n\nabla\varphi_k
=
\nabla\psi_n,
\]
we obtain
\[
\sum_{k=1}^n\int_{\T^d}\widetilde u_k r_k\,dx
=
-
\int_{\T^d}\int_{\R^d}
u(x)u(x-y)b(x-y)\cdot\nabla\psi_n(y)\,dy\,dx.
\]
We now decompose \(u\) and \(b\) in Littlewood-Paley blocks
\[
u=\sum_{j\geq1}u_j,
\qquad
b-b_0=\sum_{m\geq1}b_m,
\]
and we note that the constant term \(b_0=\fint b\) does not affect the estimates. After the change of variables \((x,y)\mapsto(x-y,-y)\) and symmetrizing in
\(j\) and \(\ell\), we are reduced to estimating, for fixed
\(n,j,m,\ell\geq1\), the interaction
\begin{equation}\label{eq:weighted-interaction}
\frac{1}{2}\int_{\T^d}
b_m\cdot
\bigl[
u_\ell\nabla(\psi_n*u_j)
+
u_j\nabla(\psi_n*u_\ell)
\bigr]\,dx.
\end{equation}
The expression is symmetric in \(j\) and \(\ell\), so we may restrict to
\(j\leq\ell\), at the cost of a factor \(2\).
By construction, we have
\[
\psi_n*u_k=
\begin{cases}
u_k, & k\leq n-1,\\
0, & k\geq n+2.
\end{cases}
\]
Since \(j\leq\ell\), if \(n\leq j-2\), both terms in~\eqref{eq:weighted-interaction} vanish. If \(n\geq\ell+1\),
then \(\psi_n\) acts as the identity on both \(u_j\) and \(u_\ell\), and hence
\[
\int_{\T^d}
b_m\cdot
\left(
u_\ell\nabla u_j+u_j\nabla u_\ell
\right)\,dx
=
\int_{\T^d}b_m\cdot\nabla(u_ju_\ell)\,dx
=
0,
\]
since \(\operatorname{div}b_m=0\). Therefore, a nonzero interaction can occur only if \(j-1\leq n\leq\ell\).

We now compare the three indices \(j,m,\ell\). Since
\[
\operatorname{spt}\widehat{u_j}
\subset
B_{2^j}(0)\setminus\overline{B_{2^{j-2}}(0)},
\]
with the analogous inclusions for \(b_m\) and \(u_\ell\), the interaction
vanishes if one of \(j,m,\ell\) exceeds the other two by at least \(3\).
Indeed, in that case the frequency carried by the largest block cannot be
canceled by the frequencies carried by the other two. Thus the two largest
indices among \(j,m,\ell\) differ by at most \(2\).

Since \(j\leq\ell\), we divide the interactions according to the relative
position of \(m\) and \(j\). If \(m\leq j-2\), the two largest indices are \(j\)
and \(\ell\), so the restriction above gives \(\ell\leq j+2\). If
\(m\geq j-1\), the same restriction gives \(|\ell-m|\leq3\). Thus all
potentially nonzero interactions fall into the following two regimes:
\[
\begin{array}{lll}
I: & m\leq j-2, & j\leq\ell\leq j+2,\\[2mm]
II: & m\geq j-1, & |\ell-m|\leq3.
\end{array}
\]
It remains to estimate the following two contributions:
\begin{align*}
I
:={}&
\sum_{j\geq3}
\sum_{\ell=j}^{j+2}
\sum_{n=j-1}^{\ell}
\left|
F(G(n))-F(G(n+1))
\right|\\
&\quad\times
\left|
\int_{\T^d}
b_{j-2}^{\leq}\cdot
\bigl[
u_\ell\nabla(\psi_n*u_j)
+
u_j\nabla(\psi_n*u_\ell)
\bigr]\,dx
\right|,
\end{align*}
and
\begin{align*}
II
:={}&
\sum_{j\geq1}
\sum_{m\geq\max\{1,j-1\}}
\sum_{\substack{\ell\geq j\\|\ell-m|\leq3}}
\sum_{n=\max\{1,j-1\}}^{\ell}
\left|
F(G(n))-F(G(n+1))
\right|\\
&\quad\times
\left|
\int_{\T^d}
b_m\cdot
\bigl[
u_\ell\nabla(\psi_n*u_j)
+
u_j\nabla(\psi_n*u_\ell)
\bigr]\,dx
\right|.
\end{align*}
\subsection{Proof of Theorem~\ref{renormalizationosgood}}
We now prove uniqueness and renormalization for solutions of the transport equation.
\begin{proof}
    For the first choice, let \(F(x)=\chi(x/R)\), where \(R>0\) and \(\chi\) is smooth, with \(\chi=1\) on \([0,1]\) and \(\chi=0\) on \([2,\infty)\). Since \(\omega\) is Osgood, \(G(k)\to\infty\) as \(k\to\infty\), so for every fixed \(R\) the sum defining \(E_F\) has only finitely many nonzero terms.
By using the mean value theorem and~\eqref{incrementG} we have
\begin{equation}\label{eq:energy-cutoff-difference}
\left|
\chi\left(\frac{G(n)}R\right)
-
\chi\left(\frac{G(n+1)}R\right)
\right|
\lesssim
\frac1R\frac{2^{-n}}{\omega(2^{-n})}.
\end{equation}
We will also use the following square-function bounds:
\begin{equation}\label{eq:energy-square-function-bounds}
\begin{aligned}
&
\left\|
\left(\sum_{k\geq1}|M(u_k)|^2\right)^{1/2}
\right\|_{L^{2p'}}
+
\left\|
\left(\sum_{k\geq1}|u_k|^2\right)^{1/2}
\right\|_{L^{2p'}}
\\
&\qquad+
\left\|
\left(\sum_{k\geq1}2^{-2k}|\nabla u_k|^2\right)^{1/2}
\right\|_{L^{2p'}}
\lesssim
\|u(t)\|_{L^{2p'}}
\lesssim
\|u\|_{L^\infty((0,T)\times\T^d)}.
\end{aligned}
\end{equation}
The first two terms are controlled by the Fefferman-Stein inequality~\eqref{FSvectorineq} and the Littlewood-Paley estimate~\eqref{LPtheorem}. For the third, observe that
\(
2^{-k}\nabla u_k
=
u*(2^{-k}\nabla\varphi_k),
\)
and apply~\eqref{LPupperbound}, since the kernels
\(2^{-k}\nabla\varphi_k\) are dyadic rescalings of a fixed mean-zero Schwartz
kernel. We start by bounding \(I\). By the fundamental theorem of calculus we have
\[
b_{j-2}^{\leq}(x)-b_{j-2}^{\leq}(x-y)
=
\int_0^1
y\cdot\nabla b_{j-2}^{\leq}(x-sy)\,ds.
\]
Using~\eqref{eq:maximaltraslate} and the fact that \(|n-j|\leq2\), we obtain
\begin{align*}
&\left|
\int_{\T^d}\int_{\R^d}
u_\ell(x)u_j(x-y)
\bigl(
b_{j-2}^{\leq}(x)-b_{j-2}^{\leq}(x-y)
\bigr)\cdot\nabla\psi_n(y)\,dy\,dx
\right|\\
&\qquad\lesssim
\int_{\T^d}
|u_\ell(x)|\,M(u_j)(x)\,
M\bigl(|\nabla b_{j-2}^{\leq}|\bigr)(x)
\int_{\R^d}
(1+2^n|y|)^{2N}|y|\,|\nabla\psi_n(y)|\,dy\,dx\\
&\qquad\lesssim
\int_{\T^d}
M(u_j)\,
M\bigl(|\nabla b_{j-2}^{\leq}|\bigr)\,
|u_\ell|\,dx,
\end{align*}
where in the last step we used the weighted kernel estimate. Using~\eqref{eq:energy-cutoff-difference} and~\eqref{eq:adjacentomega}, and summing over \(n\) and \(\ell\), we obtain
\begin{align*}
I
\lesssim
\frac1R
\int_{\T^d}
\sup_{j\geq3}
M\left(
\frac{2^{-j}}{\omega(2^{-j})}
|\nabla b_{j-2}^{\leq}|
\right)
\sum_{j\geq3}
M(u_j)
\sum_{\ell=j}^{j+2}|u_\ell|
\,dx.
\end{align*}
By Cauchy-Schwarz in \(j\) and the bounded overlap of the ranges
\(j\leq\ell\leq j+2\),
\[
\sum_{j\geq3}
M(u_j)\sum_{\ell=j}^{j+2}|u_\ell|
\lesssim
\left(\sum_{k\geq1}|M(u_k)|^2\right)^{1/2}
\left(\sum_{k\geq1}|u_k|^2\right)^{1/2}.
\]
Moreover,
\[
\sup_{j\geq3}
M\left(
\frac{2^{-j}}{\omega(2^{-j})}
|\nabla b_{j-2}^{\leq}|
\right)
\leq
M\left(
\sup_{j\geq3}
\frac{2^{-j}}{\omega(2^{-j})}
|\nabla b_{j-2}^{\leq}|
\right).
\]
Hence, by the \(L^p\) boundedness of \(M\),~\eqref{eq:adjacentomega}, and
Lemma~\ref{hajlaszdyadicbounds},
\[
\left\|
\sup_{j\geq3}
M\left(
\frac{2^{-j}}{\omega(2^{-j})}
|\nabla b_{j-2}^{\leq}|
\right)
\right\|_{L^p}
\lesssim
\|b(t,\cdot)\|_{\dot M^{\omega,p}}.
\]
Combining these estimates with Hölder's inequality and~\eqref{eq:energy-square-function-bounds}, we obtain
\[
I
\lesssim
\frac1R
\|b(t,\cdot)\|_{\dot M^{\omega,p}}
\|u\|_{L^\infty((0,T)\times\T^d)}^2.
\]
We now prove the same estimate for \(II\). Integrating the second term in~\eqref{eq:weighted-interaction} by parts and using
\(\operatorname{div}b_m=0\), we obtain
\[
\int_{\T^d}
b_m\cdot
\bigl[
u_\ell\nabla(\psi_n*u_j)
-
(\psi_n*u_\ell)\nabla u_j
\bigr]\,dx.
\]
Since \(n\geq j-1\), we have
\[
|\nabla(\psi_n*u_j)|
\lesssim
2^jM(u_j).
\]
Moreover, \(\psi_n*u_\ell\) vanishes unless \(n=\ell-1\) or \(n=\ell\),
and in these two cases we have
\[
|\psi_n*u_\ell|
\lesssim
M(u_\ell).
\]
Therefore,
\begin{align}\label{eq:weighted-case-II-interaction}
&\left|
\int_{\T^d}
b_m\cdot
\bigl[
u_\ell\nabla(\psi_n*u_j)
-
(\psi_n*u_\ell)\nabla u_j
\bigr]\,dx
\right|\notag\\
&\qquad\lesssim
2^j
\int_{\T^d}
|b_m|
\left(
|u_\ell|M(u_j)
+
M(u_\ell)2^{-j}|\nabla u_j|
\right)\,dx.
\end{align}
Before putting everything together we now sum the cutoff differences. Since \(|\ell-m|\leq3\), by~\eqref{eq:energy-cutoff-difference},
\begin{align*}
&2^j
\sum_{n=\max\{1,j-1\}}^\ell
\left|
\chi\left(\frac{G(n)}R\right)
-
\chi\left(\frac{G(n+1)}R\right)
\right|\\
&\qquad\lesssim
\frac1{R\omega(2^{-m})}
\sum_{n=\max\{1,j-1\}}^{m+3}
2^{j-n}
\frac{\omega(2^{-m})}{\omega(2^{-n})}.
\end{align*}
For \(n\leq m\),~\eqref{omegabounds} gives
\[
\frac{\omega(2^{-m})}{\omega(2^{-n})}
\lesssim
2^{-\beta(m-n)}.
\]
Hence
\[
\sum_{n=\max\{1,j-1\}}^m
2^{j-n}
\frac{\omega(2^{-m})}{\omega(2^{-n})} \lesssim 2^{-\beta(m-j)}
\sum_{n=\max\{1,j-1\}}^m
2^{-(1-\beta)(n-j)}
\lesssim
2^{-\frac{\beta}{2}(m-j)}.
\]
The terms \(n=m+1,m+2,m+3\) satisfy the same bound, since the weights are comparable by~\eqref{eq:adjacentomega}. We thus have
\[
2^j
\sum_{n=\max\{1,j-1\}}^\ell
\left|
\chi\left(\frac{G(n)}R\right)
-
\chi\left(\frac{G(n+1)}R\right)
\right|
\lesssim
\frac1R
\frac{2^{-\frac{\beta}{2}(m-j)}}{\omega(2^{-m})}.
\]
Using
\(
\frac{|b_m|}{\omega(2^{-m})}
\leq
\sup_{k\geq1}\frac{|b_k|}{\omega(2^{-k})},
\)
the preceding estimates give
\begin{align*}
II
\lesssim{}&
\frac1R
\int_{\T^d}
\sup_{k\geq1}\frac{|b_k|}{\omega(2^{-k})}
\sum_{m\geq1}
\left(
\sum_{\substack{\ell\geq1\\|\ell-m|\leq3}}
|u_\ell|
\right)
\left(
\sum_{j=1}^{m+1}
2^{-\frac{\beta}{2}(m-j)}M(u_j)
\right)\,dx\\
&+
\frac1R
\int_{\T^d}
\sup_{k\geq1}\frac{|b_k|}{\omega(2^{-k})}
\sum_{m\geq1}
\left(
\sum_{\substack{\ell\geq1\\|\ell-m|\leq3}}
M(u_\ell)
\right)
\left(
\sum_{j=1}^{m+1}
2^{-\frac{\beta}{2}(m-j)}
2^{-j}|\nabla u_j|
\right)\,dx.
\end{align*}
Since
\(
2^{-\frac{\beta}{2}r}\mathbf{1}_{\{r\geq-1\}}\in\ell^1,
\)
Cauchy-Schwarz in \(m\), the bounded overlap of the ranges
\(|\ell-m|\leq3\), and Young's convolution inequality on \(\ell^2\) give
\begin{align*}
II
\lesssim
\frac1R
\int_{\T^d}
\sup_{k\geq1}\frac{|b_k|}{\omega(2^{-k})}
\left(
\sum_{k\geq1}|M(u_k)|^2
\right)^{1/2}
\left[
\left(
\sum_{k\geq1}|u_k|^2
\right)^{1/2}
+
\left(
\sum_{k\geq1}2^{-2k}|\nabla u_k|^2
\right)^{1/2}
\right]\,dx.
\end{align*}
Lemma~\ref{hajlaszdyadicbounds}, Hölder's inequality, and~\eqref{eq:energy-square-function-bounds} then give
\[
II
\lesssim
\frac1R
\|b(t,\cdot)\|_{\dot M^{\omega,p}}
\|u\|_{L^\infty((0,T)\times\T^d)}^2.
\]
Combining the estimates for the two cases, we obtain, for almost every
\(t\in(0,T)\),
\[
\left|\frac d{dt}E_R(t)\right|
\lesssim
\frac1R
\|b(t,\cdot)\|_{\dot M^{\omega,p}}
\|u\|_{L^\infty((0,T)\times\T^d)}^2.
\]
Integrating in time gives
\[
|E_R(t)-E_R(0)|
\lesssim
\frac{\|u\|_{L^\infty((0,T)\times\T^d)}^2}{R}
\int_0^t\|b(r,\cdot)\|_{\dot M^{\omega,p}}\,dr
\]
for every \(t\in[0,T]\). Since \(\int_0^T \|b(t,\cdot)\|_{\dot M^{\omega,p}}dt<\infty\,\), the right-hand
side tends to zero as \(R\to\infty\). Since the weights converge pointwise to \(1\) and are uniformly bounded,
dominated convergence and the identity
\[
\sum_{k\geq1}\|\widetilde u_k\|_{L^2(\T^d)}^2
=
\|u\|_{L^2(\T^d)}^2
\]
give, by sending \(R\) to infinity,
\[
\|u(t)\|_{L^2(\T^d)}
=
\|u_0\|_{L^2(\T^d)}
\]
for every \(t\in[0,T]\). Since the difference of two solutions is again a solution and its \(L^2\) norm is preserved, uniqueness follows. The renormalization property can be deduced from uniqueness in several ways; see, for example,~\cite{BouchutCrippa2006}.

\end{proof}
\subsection{Proof of Theorem~\ref{regularityosgood}}
We now prove propagation of \(B^{\omega,a}\) regularity for bounded solutions of the transport equation. The proof is a modification of the preceding argument, leading to a Gronwall-type estimate similar to that of Meyer and Seis~\cite{meyer2024propagation}.

We will need the following simple equivalence between the weighted energies
associated with the Littlewood-Paley blocks and with the quadratic
Littlewood-Paley blocks. 
\begin{lem}\label{lem:quadratic-weighted-equivalence}
Let \(a\geq0\). For every integer \(N\geq1\) and every mean-free function
\(u\in L^2(\T^d)\), one has
\[
\sum_{k\geq1}
G(k\wedge N)^{2a}\|u_k\|_{L^2}^2
\sim
\sum_{k\geq1}
G(k\wedge N)^{2a}\|\widetilde u_k\|_{L^2}^2,
\]
where \(k\wedge N:=\min\{k,N\}\), and the implicit constants are independent of \(N\) and \(u\).
\end{lem}
\begin{proof}
By Lemma~\ref{lemma gk+1},
\(
G(k\wedge N)\sim G((k+1)\wedge N)
\)
uniformly in \(k\) and \(N\). Moreover, Parseval's identity gives
\begin{align*}
\sum_{k\geq1}
G(k\wedge N)^{2a}\|u_k\|_{L^2}^2
&=
\sum_{\xi\in\Z^d\setminus\{0\}}
|\widehat u(\xi)|^2
\sum_{k\geq1}
G(k\wedge N)^{2a}\widehat{\varphi_k}(\xi)^2,
\\
\sum_{k\geq1}
G(k\wedge N)^{2a}\|\widetilde u_k\|_{L^2}^2
&=
\sum_{\xi\in\Z^d\setminus\{0\}}
|\widehat u(\xi)|^2
\sum_{k\geq1}
G(k\wedge N)^{2a}\widehat{\varphi_k}(\xi).
\end{align*}
For every \(\xi\in\mathbb Z^d\setminus\{0\}\), at most two consecutive
multipliers \(\widehat{\varphi_k}(\xi)\) are nonzero. Since they are
nonnegative, sum to one, and the corresponding weights are comparable, it follows that
\[
\sum_{k\geq1}
G(k\wedge N)^{2a}\widehat{\varphi_k}(\xi)^2
\sim
\sum_{k\geq1}
G(k\wedge N)^{2a}\widehat{\varphi_k}(\xi)
\]
uniformly in \(\xi\) and \(N\).
\end{proof}
We will also need the following endpoint version of the interpolation estimate
in~\cite[Lemma~2(c)]{meyer2024propagation}, adapted to the weights
\(G(k\wedge N)\) and restricted to the two families of Littlewood-Paley
multipliers appearing in our estimates.
\begin{prop}\label{prop:interpolation}
Let \(p\in(1,\infty)\), let \(\frac12<a\leq\frac p2\), and let \(N\geq1\).
Then, for every mean-free function \(u\in L^\infty(\T^d)\),
\begin{align*}
&
\left\|
\left(
\sum_{k\geq1}
G(k\wedge N)^{2a-1}|u_k|^2
\right)^{1/2}
\right\|_{L^{2p'}}
+
\left\|
\left(
\sum_{k\geq1}
G(k\wedge N)^{2a-1}
2^{-2k}|\nabla u_k|^2
\right)^{1/2}
\right\|_{L^{2p'}}
\notag\\
&\qquad\lesssim
\|u\|_{L^\infty}^{\frac1{2a}}
\left(
\sum_{k\geq1}
G(k\wedge N)^{2a}
\|u_k\|_{L^2}^2
\right)^{\frac12-\frac1{4a}},
\end{align*}
where the implicit constant is independent of \(N\) and \(u\).
\end{prop}
\begin{proof}
The proof follows by combining a few technical results on
Triebel-Lizorkin spaces. Let \(r:=4a/(2a-1)\). By
Lemma~\ref{lemma gk+1}, the sequence
\(\{G(k\wedge N)^a\}_{k\geq1}\) is admissible in the sense of
Farkas-Leopold~\cite{FarkasLeopold2006}, uniformly in \(N\). We now apply the torus analogue of the interpolation result in~\cite[Theorem~4.10 and Remark~4.13]{Drihem2023} with interpolation
parameter \(1/(2a)\). Moreover,
\(F^0_{\infty,2}(\T^d)=\operatorname{BMO}(\T^d)\), see~\cite[Theorem~2, p.~93, and Chapter~9]{Triebel1983}, and
\(\|u\|_{\operatorname{BMO}}\lesssim\|u\|_{L^\infty}\). Thus, we obtain
\[
\left\|\left(\sum_{k\geq1}G(k\wedge N)^{2a-1}|u_k|^2\right)^{1/2}\right\|_{L^r}
\lesssim
\|u\|_{L^\infty}^{\frac1{2a}}
\left(\sum_{k\geq1}G(k\wedge N)^{2a}\|u_k\|_{L^2}^2\right)^{\frac12-\frac1{4a}}.
\]
On the other hand, by~\eqref{eq:maximaltraslate}, we have
\(2^{-k}|\nabla u_k|\lesssim M(u_k)\). Hence, by the Fefferman-Stein
inequality,
\begin{align*}
\left\|
\left(
\sum_{k\geq1}
G(k\wedge N)^{2a-1}
2^{-2k}|\nabla u_k|^2
\right)^{1/2}
\right\|_{L^r}
&\lesssim
\left\|
\left(
\sum_{k\geq1}
G(k\wedge N)^{2a-1}|M(u_k)|^2
\right)^{1/2}
\right\|_{L^r}
\\
&\lesssim
\left\|
\left(
\sum_{k\geq1}
G(k\wedge N)^{2a-1}|u_k|^2
\right)^{1/2}
\right\|_{L^r}.
\end{align*}
Finally, since \(a\leq p/2\), one has \(r\geq2p'\), and the conclusion
follows.

\end{proof}

\begin{proof}[Proof of Theorem~\ref{regularityosgood}]
We fix an integer \(N\geq1\) and consider the truncated weighted energy
\[
E_N(t)
:=
\sum_{k\geq1}
G(k\wedge N)^{2a}
\|\widetilde u_k(t)\|_{L^2}^2.
\]
The weight is constant for \(k\geq N\). Using the conservation of the
\(L^2\) norm proved above, we may write
\[
E_N(t)
=
G(N)^{2a}\|u(t)\|_{L^2}^2
+
\sum_{k=1}^{N-1}
\bigl(G(k)^{2a}-G(N)^{2a}\bigr)
\|\widetilde u_k(t)\|_{L^2}^2.
\]
The first term is constant in time, while the second is a finite sum.
Thus, as in the first part of the proof, the differentiation and the
summation by parts argument are justified, and we obtain
\[
\frac12\frac d{dt}E_N(t)
=
\sum_{n\geq1}
\left(
G(n\wedge N)^{2a}
-
G((n+1)\wedge N)^{2a}
\right)
\sum_{k=1}^n
\int_{\T^d}\widetilde u_k r_k\,dx.
\]
As in the previous proof, we will need estimates for the differences between consecutive weights and for their ratios at different scales. By the mean value theorem, Lemma~\ref{lemma gk+1}, and~\eqref{incrementG}, we have, for every \(n\geq1\),
\[
\left|
G(n\wedge N)^{2a}
-
G((n+1)\wedge N)^{2a}
\right|
\lesssim
G(n\wedge N)^{2a-1}
\frac{2^{-n}}{\omega(2^{-n})}.
\]
We will also need to bound the ratio of the weights at different scales. By iterating Lemma~\ref{lemma gk+1}, we have, uniformly in \(N\),
\[
\frac{G(m\wedge N)}{G(j\wedge N)}
\lesssim
1+m-j
\]
for all integers \(1\leq j\leq m\).
We first estimate \(I\). Combining the low-frequency interaction estimate
with the preceding bound on the differences of the weights, we obtain
\begin{align*}
I
\lesssim
\int_{\T^d}
\sup_{j\geq3}
M\left(
\frac{2^{-j}}{\omega(2^{-j})}
|\nabla b_{j-2}^{\leq}|
\right)
\sum_{j\geq3}
G(j\wedge N)^{2a-1}M(u_j)
\sum_{\ell=j}^{j+2}|u_\ell|
\,dx.
\end{align*}
By Cauchy-Schwarz in \(j\), the comparability of the weights at neighboring
scales, and the bounded overlap of the ranges \(j\leq\ell\leq j+2\), we have
\begin{align*}
&
\sum_{j\geq3}
G(j\wedge N)^{2a-1}M(u_j)
\sum_{\ell=j}^{j+2}|u_\ell|
\\
&\qquad\lesssim
\left(
\sum_{k\geq1}
G(k\wedge N)^{2a-1}|M(u_k)|^2
\right)^{1/2}
\left(
\sum_{k\geq1}
G(k\wedge N)^{2a-1}|u_k|^2
\right)^{1/2}.
\end{align*}
Therefore, by Hölder's inequality, the \(L^p\) boundedness of \(M\), the
Fefferman-Stein inequality,~\eqref{eq:adjacentomega}, and
Lemma~\ref{hajlaszdyadicbounds}, we obtain
\[
I
\lesssim
\|b(t,\cdot)\|_{\dot M^{\omega,p}}
\left\|
\left(
\sum_{k\geq1}
G(k\wedge N)^{2a-1}|u_k|^2
\right)^{1/2}
\right\|_{L^{2p'}}^2.
\]
We now turn to \(II\). To sum the differences of the weights, we argue in a similar way to the previous proof. Since \(|\ell-m|\leq3\), the preceding bound on the differences of the
weights gives
\begin{align*}
&2^j
\sum_{n=\max\{1,j-1\}}^\ell
\left|
G(n\wedge N)^{2a}
-
G((n+1)\wedge N)^{2a}
\right|\\
&\qquad\lesssim
\frac{1}{\omega(2^{-m})}
\sum_{n=\max\{1,j-1\}}^{m+3}
G(n\wedge N)^{2a-1}
\frac{\omega(2^{-m})}{\omega(2^{-n})}
2^{j-n}.
\end{align*}
Arguing as in the first part of the proof, and using the monotonicity of \(G\), we obtain
\[
2^j
\sum_{n=\max\{1,j-1\}}^\ell
\left|
G(n\wedge N)^{2a}
-
G((n+1)\wedge N)^{2a}
\right|
\lesssim
\frac{G(m\wedge N)^{2a-1}}{\omega(2^{-m})}
2^{-\frac{\beta}{2}(m-j)}.
\]
For \(m\geq j\), the preceding estimate on the ratio of the weights gives
\[
\left(
\frac{G(m\wedge N)}{G(j\wedge N)}
\right)^{a-\frac12}
\lesssim
(1+m-j)^{a-\frac12}
\lesssim
2^{\frac{\beta}{4}(m-j)}.
\]
Therefore,
\[
G(m\wedge N)^{2a-1}
2^{-\frac{\beta}{2}(m-j)}
\lesssim
G(m\wedge N)^{a-\frac12}
G(j\wedge N)^{a-\frac12}
2^{-\frac{\beta}{4}(m-j)}.
\]
The same estimate holds for \(m=j-1\), up to a change in the implicit
constant. Combining these estimates with~\eqref{eq:weighted-case-II-interaction}, and arguing as in the previous proof, we use
\(
2^{-\frac{\beta}{4}r}\mathbf{1}_{\{r\geq-1\}}
\in\ell^1(\mathbb{Z})
\)
together with Cauchy-Schwarz in \(m\), the bounded overlap of the ranges
\(|\ell-m|\leq3\), and Young's convolution inequality on \(\ell^2\) to obtain
\begin{align*}
II
\lesssim
\int_{\T^d}
&\sup_{k\geq1}
\frac{|b_k|}{\omega(2^{-k})}
\left(
\sum_{k\geq1}
G(k\wedge N)^{2a-1}|M(u_k)|^2
\right)^{1/2}\\
&\times
\left[
\left(
\sum_{k\geq1}
G(k\wedge N)^{2a-1}|u_k|^2
\right)^{1/2}
+
\left(
\sum_{k\geq1}
G(k\wedge N)^{2a-1}
2^{-2k}|\nabla u_k|^2
\right)^{1/2}
\right]\,dx.
\end{align*}
By the Fefferman-Stein inequality, Lemma~\ref{hajlaszdyadicbounds},
and Hölder's inequality, we conclude that
\begin{align*}
II
\lesssim{}&
\|b(t,\cdot)\|_{\dot M^{\omega,p}}
\left\|
\left(
\sum_{k\geq1}
G(k\wedge N)^{2a-1}|u_k|^2
\right)^{1/2}
\right\|_{L^{2p'}}\\
&\times
\left[
\left\|
\left(
\sum_{k\geq1}
G(k\wedge N)^{2a-1}|u_k|^2
\right)^{1/2}
\right\|_{L^{2p'}}
+
\left\|
\left(
\sum_{k\geq1}
G(k\wedge N)^{2a-1}
2^{-2k}|\nabla u_k|^2
\right)^{1/2}
\right\|_{L^{2p'}}
\right].
\end{align*}
If \(a>\frac12\), applying Proposition~\ref{prop:interpolation} and
Lemma~\ref{lem:quadratic-weighted-equivalence} to the estimates for \(I\) and
\(II\), we obtain
\[
\left|
\frac{d}{dt}E_N(t)
\right|
\lesssim
\|b(t,\cdot)\|_{\dot M^{\omega,p}}
\|u\|_{L^\infty}^{1/a}
E_N(t)^{1-\frac{1}{2a}}.
\]
If \(a=\frac12\), the weights appearing in the square functions are equal to
one, and hence~\eqref{eq:energy-square-function-bounds} gives
\[
\left|
\frac{d}{dt}E_N(t)
\right|
\lesssim
\|b(t,\cdot)\|_{\dot M^{\omega,p}}
\|u\|_{L^\infty}^2.
\]
Thus, the first estimate continues to hold also when \(a=\frac12\). Rewriting the differential inequality as
\[
\frac{d}{dt}E_N(t)^{\frac1{2a}}
\lesssim
\|b(t,\cdot)\|_{\dot M^{\omega,p}}
\|u\|_{L^\infty}^{\frac1a},
\]
and integrating in time, we obtain, by
Lemma~\ref{lem:quadratic-weighted-equivalence}, uniformly in \(N\),
\[
\left(
\sum_{k\geq1}
G(k\wedge N)^{2a}\|u_k(t)\|_{L^2}^2
\right)^{1/2}
\lesssim
\left(
\sum_{k\geq1}
G(k\wedge N)^{2a}\|u_k(0)\|_{L^2}^2
\right)^{1/2}
+
\|u\|_{L^\infty}
\left(
\int_0^t
\|b(r,\cdot)\|_{\dot M^{\omega,p}}\,dr
\right)^a.
\]
Letting \(N\to\infty\), we conclude that
\[
\|u(t)\|_{B^{\omega,a}}
\lesssim
\|u_0\|_{B^{\omega,a}}
+
\|u\|_{L^\infty}
\left(
\int_0^t
\|b(r,\cdot)\|_{\dot M^{\omega,p}}\,dr
\right)^a.
\]

\end{proof}

\printbibliography
\end{document}